%% file: main_2.tex
\documentclass[10pt, english,oneside,reqno]{amsart}
\usepackage[utf8]{inputenc}
\usepackage[english]{babel}
\usepackage{hyperref}
\usepackage{color}
\usepackage{enumerate}
\usepackage{fancyhdr}

\usepackage{amsmath}
\usepackage{amscd}
\usepackage{amsfonts}
\usepackage{amsthm}
\usepackage{amstext}
\usepackage{amssymb}
\usepackage{amsbsy}
\usepackage{mathrsfs}
\usepackage{tikz-cd}
\usepackage{mathtools}
\usepackage{enumitem}
\newtheorem*{corollary*}{Corollary}

\usepackage{float}
\usepackage{caption}
\usepackage{subcaption}
\usepackage{graphicx}

\usepackage[alphabetic,abbrev]{amsrefs}

\usepackage{url}

\theoremstyle{plain}
\newtheorem{theorem}{Theorem}[section]

\newtheorem{proposition}[theorem]{Proposition}
\newtheorem{corollary}[theorem]{Corollary}
\newtheorem{lemma}[theorem]{Lemma}

\theoremstyle{remark}

\newtheorem{remark}[theorem]{Remark}
\numberwithin{equation}{section}
\numberwithin{figure}{section}

\definecolor{esperance}{rgb}{0.1,0.5,0.3}

\newcommand{\Prob}{\mathbb{P}}

\newcommand{\Q}{\mathbb Q}
\newcommand{\Z}{\mathbb Z}
\newcommand{\cP}{\mathcal P}

\hypersetup{
 colorlinks=false,
 pdfborder={0 0 1},
 linkbordercolor={1 0 0},
 citebordercolor={0 1 0},
 urlbordercolor={1 0 0},
 pdftitle={The Odlyzko-Poonen conjecture},
 pdfauthor={}
}
\renewcommand{\eprint}[1]{\href{https://arxiv.org/abs/#1}{\nolinkurl{arXiv:#1}}}
\renewcommand{\PrintDOI}[1]{DOI \href{https://doi.org/#1}{\nolinkurl{#1}}}

\title[The Odlyzko--Poonen conjecture]{The Odlyzko--Poonen conjecture on irreducibility of random polynomials}
\date{}

\begin{document}
\author{Constantin Kogler}
\email{kogler@ias.edu}
\address{Constantin Kogler, Institute for Advanced Study, 1 Einstein Dr, Princeton, NJ 08540, United States of America}
\begin{abstract}
The Odlyzko--Poonen conjecture states that a monic polynomial of constant coefficient $1$ and with remaining coefficients chosen independently and uniformly from $\{0,1 \}$ is irreducible over the rational numbers with probability tending to one as the degree tends to infinity. We prove the Odlyzko--Poonen conjecture unconditionally.
\end{abstract}
\maketitle
\thispagestyle{empty}

\input{sections/introduction.tex}
\input{sections/autocorrelations-modulo-four.tex}
\input{sections/reciprocal-factors-and-sharp-rate.tex}

\input{bibliography/references.tex}
\end{document}

%% file: sections/introduction.tex
\section{Introduction}
For $n \geq 1$, let
\[
\cP_n=\left\{1+\sum_{j=1}^{n-1}p_jx^j+x^n:p_j\in\{0,1\}\right\}
\]
and choose $P_n$ uniformly at random from the $2^{n-1}$ elements of $\cP_n$. Odlyzko and Poonen
\cite{OP93} conjectured that $P_n$ is irreducible over $\Q$ with
probability tending to one as the degree tends to infinity. In 1999, Konyagin \cite{Kon99} established the
lower bound $c/\log n$ for the probability of irreducibility. Twenty years later, Breuillard and Varj\'u \cite{BV19} proved
the conjecture assuming the Riemann hypothesis for Dedekind zeta functions and also established results for general coefficient measures. Bary-Soroker, Koukoulopoulos and Kozma \cite{BSKK23} showed
unconditional results for coefficients uniformly sampled from at least 35 consecutive integers. Moreover, for general coefficient measures including the case of $0,1$ polynomials \cite{BSKK23} showed that, for some $\theta,c>0$, the probability of a
factor of degree at most $\theta n$ is $O(n^{-c})$. Experimental evidence was discussed in \cite{BBBSWW18}.

In this paper, we establish the Odlyzko-Poonen conjecture without any assumptions. The proof is formalized in Lean.

\begin{theorem}\label{thm:main} As $n \to \infty$,
\[
\Prob(P_n\text{ is irreducible over }\Q)\longrightarrow1.
\]
Moreover, as $n\to\infty$,
\begin{equation}\label{eq:asymptoticexpansion}
\Prob(P_n\text{ is reducible over }\Q)
=\sqrt{\frac{2}{\pi n}}+O(n^{-1}).
\end{equation}
\end{theorem}

The asymptotic expansion \eqref{eq:asymptoticexpansion} was already established conditional on the Generalized Riemann Hypothesis in \cite{BV19}*{Corollary 4} together with a formula for higher order terms \cite{BV19}*{Corollary 3}. As discussed in Remark \ref{rm:higherorderexpansion}, our proof also allows us to determine the probability of reducibility to arbitrary order. 

In \cite{BV19}*{Theorem 2} it is moreover shown assuming the Generalized Riemann Hypothesis that there exist $c,C > 0$ such that with probability at least $1 - Ce^{-c\sqrt{n}/\log n}$ a uniformly random $P_n \in \mathcal{P}_n$ has a decomposition $P_n = \Phi I$ with $\Phi$ a product of cyclotomic polynomials satisfying $\deg \Phi \leq C\sqrt{n}$ and $I \in \mathbb{Z}[x]$ being irreducible and noncyclotomic. We also make the latter result unconditional and strengthen the probability from $1 - Ce^{-c\sqrt{n}/\log n}$ to $1 - Ce^{-c\sqrt{n}}$. Indeed, we prove an almost exponential bound on the existence of such a decomposition $P_n = \Phi I$ with no degree condition on $\Phi$.

\begin{theorem}(Irreducibility of the noncyclotomic part)\label{thm:IrredFac} There exist absolute constants \(c,C>0\) such that for every \(n\ge1\) a uniformly random \(P_n\in\mathcal P_n\) admits a factorization \(P_n=\Phi I\), where \(\Phi\) is a product of cyclotomic polynomials and \(I\in\mathbb Z[x]\) is irreducible and noncyclotomic, with probability at least
\[
1-C\exp\!\left(-\frac{cn}{(\log n)^4}\right).
\]
Moreover, for every fixed \(0<\alpha<1\), there exist constants \(c_\alpha,C_\alpha>0\) such that, with probability at least \(1-C_\alpha e^{-c_\alpha n^\alpha}\), such a factorization exists with \(\deg \Phi \le n^\alpha\).
\end{theorem} 

We proceed with outlining the proof of Theorem~\ref{thm:main} and Theorem~\ref{thm:IrredFac}. The starting point is what we call the twin factorization trick. The twin factorization trick has its origin in the work of Ljunggren \cite{Lju60}, dealing with certain classes of polynomials.  A generalization of  Ljunggren's ideas to arbitrary polynomials can be found in Filaseta \cite{Fil99}. Our proof combines the twin factorization trick with the insight that non-trivial twin factorizations are rare.

For a polynomial $P(x)= \sum_{j = 0}^n a_j x^j$ of degree $n$, denote by $$P^{*}(x) = x^{\deg P} P(x^{-1}) = \sum_{j = 0}^n a_{n-j} x^j$$ the reciprocal of $P$, that is the polynomial with reversed coefficients. We will work with the reciprocal product $PP^*$. Note that for $-n\le k\le n$, the coefficient of $x^{n-k}$ in $PP^*$ is
$\sum_{j=0}^{n-|k|}a_j a_{j+|k|}$.

We now explain the twin factorization trick. Assume there is a factorization $P_n = AB$ with $A,B \in \mathbb{Z}[x]$ monic polynomials, which is sufficient to consider by Gauss's Lemma. One then considers the twin polynomial $Q = AB^{*}$ and observes that since taking the reciprocal commutes with polynomial products and as polynomial multiplication is commutative, we have the twin factorization $QQ^* = P_nP_n^*$ of the reciprocal product $P_nP_n^*$. It is known from \cite{Fil99} (see Lemma~\ref{lem:flip}) that $Q$ is also in $\mathcal{P}_n$ and that, if neither $A$ nor $B$ is reciprocal, $Q$ is different from $P_n$ and $P_n^{*}$. In other words, if $P_n$ is reducible and has no non-constant reciprocal factors, we have a non-trivial twin factorization.

The occurrence of reciprocal noncyclotomic factors can be addressed with known methods and a nearly exponential bound will be established in Proposition~\ref{lem:reciprocal}. The core of the proof of Theorem~\ref{thm:main} and Theorem~\ref{thm:IrredFac} is showing that the probability of a non-trivial twin factorization  $QQ^{*} = P_nP_n^{*}$ with $Q \in P_n\setminus\{P_n,P_n^*\}$ is small, as is established in Theorem~\ref{thm:rigidity} below with the weaker requirement that $QQ^{*} \equiv P_nP_n^{*} \pmod 4$. To then deduce Theorem~\ref{thm:main} one applies standard methods to deal with cyclotomic factors, while Theorem~\ref{thm:IrredFac} relies on general estimates on the occurrence of high degree factors.

\begin{theorem}\label{thm:rigidity}
For $P_n$ uniform in $\cP_n$,
\begin{equation}\label{eq:rigidity} 
\Prob\bigl(\exists Q\in\cP_n\setminus\{P_n,P_n^*\}:
 QQ^*\equiv P_nP_n^* \pmod 4\bigr)\le8(3/4)^{\lfloor(n-1)/4\rfloor}.
\end{equation}
\end{theorem} 

To the author's knowledge, Theorem~\ref{thm:rigidity} is novel. We remark that the exponential bound of Theorem~\ref{thm:rigidity} is not necessary for Theorem~\ref{thm:main} and any bound of the size $O(n^{-1})$ would be sufficient. We highlight the work of Filaseta-Kalogirou, who formulated Conjecture 1 in \cite{FK26} on difference multisets and stated that the Odlyzko-Poonen conjecture implies this conjecture. As explained in Section~\ref{sec:FK}, Theorem~\ref{thm:rigidity} implies Conjecture 1 of \cite{FK26} and establishes an exponential error rate. Moreover, it can be shown that an exponential error rate for Conjecture 1 of \cite{FK26} implies an exponentially small probability of a
non-trivial twin factorization. Our proof can therefore be viewed as completing a strategy closely related to ideas of Filaseta-Kalogirou.

Our results also apply to one-dimensional binary phase retrieval as discussed in \cite{JOH17}, \cite{YW18} and \cite{WLMZ20}. For \(p\in\{0,1\}^{n+1}\), the Fourier magnitude of $p$ is \(M_p(\theta)=|\sum_{j=0}^{n}p_je^{-ij\theta}|=|P(e^{-i\theta})|\), where \(P(z)=\sum_{j=0}^{n}p_jz^j\) and \(\theta\in[0,2\pi)\). For \(P,Q\in\mathcal P_n\), equality of Fourier magnitudes is equivalent to \(PP^*=QQ^*\), which, when \(P\) is irreducible, implies \(Q=P\) or \(Q=P^*\). Thus, Theorem~\ref{thm:main} gives uniqueness up to reversal with probability \(1-O(n^{-1/2})\) for a uniformly random binary signal with endpoints \(1\). Theorem~\ref{thm:rigidity} improves this uniqueness error rate to \(8(3/4)^{\lfloor(n-1)/4\rfloor}\), even when only the residues modulo \(4\) of the aperiodic autocorrelations \(r_k(p)=\sum_{j=0}^{n-k}p_jp_{j+k}\) with \(0\le k\le n\) are given. 

In this paper we focus on the case of $0,1$ polynomials and Theorem~\ref{thm:rigidity} is specific to this setting. The twin factorization trick, however, generalizes to arbitrary coefficient measures and we will treat more general cases in future work.

\subsection*{Machine Contribution Statement} The proofs are due to GPT-6 Astra. The author rewrote and checked the arguments, takes responsibility for their correctness
and would be delighted to give talks on the results and to answer any questions.

A proof of Theorem~\ref{thm:main} together with Theorem~\ref{thm:rigidity} was first found on September 14th, 2026 at noon, after the author was studying Bernoulli convolutions and the model proved that the Bernoulli convolution with defining parameter being a  zero in $(1/2,1)$ of a random polynomial with $-1,0,1$ coefficients has dimension 1. The model initially proved the latter claim conditional on the Generalized Riemann Hypothesis, connecting to the work of Breuillard-Varjú \cite{BV19}, and was subsequently able to prove the claim unconditionally. This success led the author to try the Odlyzko-Poonen conjecture. Theorem~\ref{thm:IrredFac} resulted from refinements of the proof. 

\subsection*{Lean Formalization} GPT-6 Astra formalized all results from this paper and all necessary results from previous work. The formalization took around 30 hours and around 22,000 lines of code were written. The Lean source code is available in the accompanying repository \cite{Kog26Lean}.

\subsection*{Acknowledgments} I thank Emmanuel Breuillard and Péter Varjú for their rapid, selfless help and for closely supporting the writing of this paper. I am very grateful to Emmanuel Breuillard for checking the proof, his insightful mathematical comments and for extensively discussing the previous literature. I am indebted to Michael Filaseta for his thorough reading of the paper, highly detailed comments, and insights into simplifications and extensions. I thank Andrew Odlyzko for historical remarks and Akshay Venkatesh for advice and optimism. The author holds a Postdoc Mobility Fellowship from the Swiss National Science Foundation (grant number
235409) and thanks the Institute for Advanced Study.

%% file: sections/autocorrelations-modulo-four.tex
\section{Proof of Theorem~\ref{thm:rigidity}}
For $f\in\mathbb F_2[x]$, let $s(f)\in\Z[x]$ be its coefficientwise
zero-one lift. Denote for $n\geq 1$ by $\mathcal{B}_n$ the image of $\mathcal{P}_n$ under the natural map $\mathbb{Z}[x] \to \mathbb{F}_2[x]$, that is
\[
\mathcal B_n=\{f\in\mathbb F_2[x]:\deg f=n,\ f\text{ monic},\ f(0)=1\},
\qquad |\mathcal B_n|=2^{n-1}.
\]
The following proposition, which is the core novelty of our proof, implies Theorem~\ref{thm:rigidity}.

\begin{proposition}\label{prop:pairs}
Let $1\le d\le e$ and choose $a\in\mathcal B_d$ and $b\in\mathcal B_e$ independently and
uniformly. If $P=s(ab)$ and $Q=s(ab^*)$, then
\[
\Prob(a\ne a^*,\ PP^*\equiv QQ^*\pmod4)
\le2(3/4)^{\lfloor(e-1)/2\rfloor}.
\]
\end{proposition}

\begin{proof}
Write $a(x) = \sum_{i = 0}^d a_i x^i$ and $b(x) = \sum_{i = 0}^e b_i x^i$. Furthermore, denote
\[
n=d+e,\qquad m=\left\lfloor\frac{e-1}{2}\right\rfloor,
\]
and let $E$ be the event $PP^*\equiv QQ^*\pmod 4$.
For a polynomial $f$, write $[f]_k$ for its coefficient of $x^k$.

We first describe a useful way of sampling $b$. Set
\[
c=b+b^*,\qquad c_i=b_i+b_{e-i}.
\]
The sums \(c_i=b_i+b_{e-i}\) in $\mathbb{F}_2$ for $1 \leq i \leq m$ are independent fair 0-1 random variables.
Moreover, conditional on $c$, the coefficients $b_1,\ldots,b_m$ are still
independent and fair, and the opposite coefficients are determined by $b_{e-i}=b_i+c_i.$
If $e$ is even, the middle coefficient of $b$ is an additional
independent variable and it will not enter the calculations below.
Thus we may first choose $c$, and then choose the coefficients of $b$.

So fix $a$ and $c$. Modulo $2$, both $PP^*$ and $QQ^*$ equal
$aa^*bb^*$. We can therefore define
\[
D_k=\frac{[PP^*-QQ^*]_k}{2}\pmod 2.
\]
The event $E$ requires every $D_k$ to vanish and in particular $E\subseteq\{D_1=\cdots=D_m=0\}.$ We shall check these first $m$ conditions by revealing
$b_1,\ldots,b_m$ in order.

Write $P=\sum_{i=0}^n p_i x^i$ and $Q=\sum_{i=0}^n q_i x^i.$
Since the leading and constant coefficients of $P$ are $1$, for $1\leq k \leq m$,
\[
[PP^*]_k
=
p_k+p_{n-k}
+\sum_{i=1}^{k-1}p_i p_{n-k+i}.
\]
For $k\le m$ and with $a$ fixed, the coefficients appearing here depend only on the
first $k$ opposite pairs of coefficients $(b_1, b_{e-1}), \ldots , (b_k, b_{e-k})$ of $b$. The same holds for
$Q$. Therefore, with $a$ and $c$ fixed, $D_k$ depends only on
$b_1,\ldots,b_k$.

We now determine whether changing the new coefficient $b_k$ changes $D_k$.
Changing $b_k$ while keeping $c$ fixed also changes $b_{e-k}$.
Because $2k<e$, among the coefficients appearing in the expression
for $[PP^*]_k$, only $p_k$ and $p_{n-k}$ change. So if we change $b_k$ with $a$ fixed we also change $p_k \equiv b_k + \sum_{r = 1}^{\min(d,k)} a_r b_{k-r} \pmod 2$ and similarly $p_{n - k}$ so that
\[
p_k\longmapsto 1-p_k,\qquad
p_{n-k}\longmapsto 1-p_{n-k}.
\]
The same holds for $q_k$ and $q_{n-k}$. Hence the change in the integer
quantity $[PP^*-QQ^*]_k/2$ is
\[
\frac{
(2-2p_k-2p_{n-k})-(2-2q_k-2q_{n-k})
}{2}.
\]
Reducing modulo $2$, we find that changing $b_k$ changes $D_k$ by the image in $\mathbb{F}_2$ of
\[
[P + P^{*} +  Q + Q^{*}]_k = p_k+p_{n-k}+q_k+q_{n-k}.
\]

On the other hand, the change of $D_k$ is determined by $a$ and $c$. Indeed, let a bar denote reduction modulo $2$. We have $\overline P+\overline Q=a(b+b^*)=ac$
and $\overline{P^*}+\overline{Q^*}
=a^*(b^*+b)=a^*c.$
Thus, by the above, changing $b_k$ changes $D_k$ by
\[
\lambda_k:=[(a+a^*)c]_k\in\mathbb F_2.
\]
In particular, $\lambda_k$ is determined by $a$ and $c$.
 
If $\lambda_k=1$, changing $b_k$ flips $D_k$, so exactly one of
the two choices of $b_k$ satisfies $D_k=0$.
If $\lambda_k=0$, changing $b_k$ leaves $D_k$ unchanged, so either
both choices satisfy the condition or neither does.

Conditional on $a$ and $c$ and all previously revealed coefficients, 
$b_k$ is a fresh fair coefficient. Revealing the coefficients successively therefore
gives
\[
\mathbb P(E\mid a,c)
\le
\mathbb P(D_1=\cdots=D_m=0\mid a,c)
\le
\prod_{k=1}^m 2^{-\lambda_k},
\]
where each $\lambda_k$ is identified with its representative in
$\{0,1\}$. 

Next, keep $a\ne a^*$ fixed and average over $c$. Write
\[
r=\sum_{i = 0}^d r_i x^i  =a+a^*\quad\text{ and }\quad
j=\min\{i:r_i\ne0\}.
\]
Equivalently, $j$ is the first index at which $a_j\ne a_{d-j}$.
Thus $1\le j\le
\left\lfloor\frac{d-1}{2}\right\rfloor
\le m$ and $r_j=1.$ Since $c_0=0$, the first $j+1$ coefficients
$\lambda_0,\lambda_1,\ldots,\lambda_j$ vanish. The subsequent coefficients
have the form $\lambda_{j+1}=c_1$ and $\lambda_{j+2}=c_2+r_{j+1}c_1$ as well as in general, for $1\le t\le m-j$,
\[
\lambda_{j+t}
=
c_t+\sum_{\ell=1}^{t-1}r_{j+t-\ell}c_\ell.
\]
Thus each successive $\lambda_{j+t}$ contains a new fair coefficient $c_t$
with coefficient $1$. Equivalently, the map $(c_1,\ldots,c_{m-j})
\mapsto
(\lambda_{j+1},\ldots,\lambda_m)$
comes from a triangular matrix with diagonal entries equal to $1$, and is therefore
a bijection. It follows that
$\lambda_{j+1},\ldots,\lambda_m$ are independent fair $0$ or $1$ coefficients.

For a fair coefficient $\lambda$, we have
$\mathbb E[2^{-\lambda}]
=
\frac12\cdot1+\frac12\cdot\frac12
=
\frac34.$
Averaging the preceding conditional bound over $c$, we obtain $\mathbb P(E\mid a)
\le
\mathbb E_c\left[\prod_{k=1}^m2^{-\lambda_k}\right]
=
(3/4)^{m-j}.$ Finally, we average over $a$ and denote $h=\left\lfloor\frac{d-1}{2}\right\rfloor.$
The $h$ opposite pairs $(a_i,a_{d-i})$, for $1\le i\le h$, are
independent, and each pair agrees with probability $1/2$ and differs
with probability $1/2$. Consequently, the unconditional probability
that the first differing pair has index $j$ is $\left(\frac12\right)^{j-1}\frac12=2^{-j}$ with $1 \leq j \leq h$.
If no pair differs, then $a=a^*$, which is excluded from the event
under consideration. Therefore
\[
\begin{aligned}
\mathbb P(a\ne a^*,E)
\le
\sum_{j=1}^h 2^{-j}(3/4)^{m-j} = (3/4)^m\sum_{j=1}^h(2/3)^j \le
2(3/4)^m.
\end{aligned}
\]
This proves the proposition.
\end{proof}

We deduce Theorem~\ref{thm:rigidity} from Proposition~\ref{prop:pairs}.

\begin{proof}[Proof of Theorem~\ref{thm:rigidity}]
For $Q \in \mathcal{P}_n \backslash \{ P_n, P_n^* \}$ with $QQ^* \equiv P_nP_n^* \pmod 4$, write $F=\bar P_n$ and $G=\bar Q$ with the bar denoting reduction $\mathrm{mod} \,\, 2$. Denote $a=\gcd(F,G)$  so that $F = ab$ and $G = at$ with $b,t \in \mathbb{F}_2[x]$. Canceling $aa^*$ from $FF^*=GG^*$ gives
$bb^*=tt^*$. Since $\gcd(b,t)=1$, we have $b\mid t^*$ and equal degrees
and monicity give $t=b^*$. Hence $P_n=s(ab)$ and $Q=s(ab^*).$ Both $a$ and $b$ are nonreciprocal, since otherwise $Q=P_n^*$ or $Q=P_n$.
Replacing $Q$ by $Q^*=s(ba^*)$ exchanges $a$ and $b$. So we may arrange
$1\le d=\deg a\le e=\deg b$ with $d+e=n$.

For a fixed degree split $d + e = n$ there are $2^{d-1}2^{e-1} = 2^{n-2}$ ordered pairs $(a,b)$ of these degrees. By Proposition 2.1, at most
$2^{n-2}\cdot 2(3/4)^{\lfloor(e-1)/2\rfloor}$
of these pairs satisfy \(a\ne a^*\) and the required congruence for \(P_n=s(ab)\) and \(Q=s(ab^*)\).   Thus, dividing by \(|\mathcal P_n|=2^{n-1}\), the probability contributed by this degree split is at most
\[
\frac{2^{n-2}\cdot 2(3/4)^{\lfloor(e-1)/2\rfloor}}{2^{n-1}}
=(3/4)^{\lfloor(e-1)/2\rfloor}.
\]Summing over the possible degree splits therefore gives
\[
\mathbb P\!\left(
\exists Q\in\mathcal P_n\setminus\{P_n,P_n^*\}:
QQ^*\equiv P_nP_n^*\pmod 4
\right)
\le
\sum_{e=\lceil n/2\rceil}^{n-1}
(3/4)^{\lfloor(e-1)/2\rfloor}.
\]
Each exponent in this sum is at least \(\lfloor(n-1)/4\rfloor\), and each integer exponent occurs at most twice. Consequently, the sum is bounded by
$2\sum_{k=\lfloor(n-1)/4\rfloor}^{\infty}(3/4)^k
=8(3/4)^{\lfloor(n-1)/4\rfloor},$
which proves \eqref{eq:rigidity}.
\end{proof}

\subsection{Application to Conjecture 1 of \cite{FK26}}\label{sec:FK}

We show how Theorem~\ref{thm:rigidity} results in a quantitative form of the difference-multiset conjecture of Filaseta and Kalogirou \cite{FK26}*{Conjecture 1}. For a subset $A \subset \{0,1,\ldots , n\}$ we denote by $A - A$ the set of differences as a multiset. As explained in \cite{FK26}*{Section 5}, the Odlyzko-Poonen conjecture implies Conjecture 1 from \cite{FK26}. We show that Theorem~\ref{thm:rigidity} implies a strong error rate for the latter conjecture. For simplicity, we formulate our claim in comparison to \cite{FK26} with the additional requirement that $n \in A$. An analogous rate for the precise formulation of Conjecture 1 from \cite{FK26} can easily be deduced.

\begin{corollary}
    Denote by \(M_n=|\{A-A:A\subseteq\{0,\ldots,n\},\ 0,n\in A\}|\) the number of difference multisets. Then it holds that \[M_n=2^{n-2}+O(12^{n/4}).\]
\end{corollary}

\begin{proof} Associate each \(A\) with \(P_A(x)=\sum_{a\in A}x^a\). Equality of difference multisets is equivalent to \(P_AP_A^*=P_BP_B^*\). Setting \(\varepsilon_n=8(3/4)^{\lfloor(n-1)/4\rfloor}\), Theorem~\ref{thm:rigidity} shows that at least \(2^{n-1}(1-\varepsilon_n)\) sets \(A\) are determined by \(A-A\) up to reflection. These sets therefore yield at least \(2^{n-2}(1-\varepsilon_n)\) distinct difference multisets. Conversely, reflection preserves the difference multiset and exactly \(2^{\lfloor n/2\rfloor}\) sets are fixed by reflection. Hence \(2^{n-2}(1-\varepsilon_n)\le M_n\le \frac{2^{n-1} + 2^{\lfloor n/2\rfloor}}{2}= 2^{n-2}+2^{\lfloor n/2\rfloor-1}\), which proves the estimate as $2^{n-2}\varepsilon_n = O(12^{n/4})$ and $2^{\lfloor n/2\rfloor-1} = O(12^{n/4})$. 
\end{proof}

%% file: sections/reciprocal-factors-and-sharp-rate.tex
\section{Proof of Theorem~\ref{thm:main} and Theorem~\ref{thm:IrredFac}}

Our first lemma follows from \cite{Fil99}*{Lemma 1} and \cite{Fil99}*{Lemma 3}. We include a proof for completeness.

\begin{lemma}\label{lem:flip}
If $P\in\cP_n$ and $P=AB$ with $A,B\in\Z[x]$ monic non-constant polynomials, then it holds that $A(0)=B(0)=1$ and that the twin polynomial $Q=AB^*$ is in $\cP_n$. We also have the twin factorization $QQ^* = PP^*$.

Moreover, if neither factor $A$ nor $B$ is reciprocal, then the twin factorization is non-trivial, that is $Q$ is different from $P$ and $P^{*}$. More precisely, $Q \in \{ P,P^{*}\}$ if and only if either $B = B^{*}$ or $A = A^{*}$. 
\end{lemma}
\begin{proof}
The integers $A(0)$ and $B(0)$ multiply to one, so they are either
both $1$ or both $-1$. In the case that both are $-1$, it holds that $A(0)<0$, while the polynomial being monic
gives $A(t)>0$ for sufficiently large $t$. By continuity, $A$, and
hence $P$, would have a positive root, contrary to $P(t)>0$ for
$t\ge0$. Thus $A(0)=B(0)=1$. It follows that $Q=AB^*$ is a monic integer polynomial of degree $n$
with constant coefficient one. 

Since reversal respects products and
$(B^*)^*=B$, we have by commutativity of polynomial multiplication
\[
Q^*=A^*B,\qquad
QQ^*=(AB^*)(A^*B)=(AB)(A^*B^*)=PP^*.
\]

Write $P=\sum_{i=0}^n p_i x^i$ and $Q=\sum_{i=0}^n q_i x^i$. As $B^*(1)=B(1)$ we have $Q(1)=P(1)$ and therefore  $\sum_{i=0}^n p_i
=\sum_{i=0}^n q_i $. Looking at the coefficient of $x^n$ of $QQ^*=PP^*$, we have $\sum_{i=0}^n p_i^2
=\sum_{i=0}^n q_i^2$. As moreover $p_i^2=p_i$ it follows that
\[
\sum_{i=0}^n q_i^2
=\sum_{i=0}^n p_i^2
=\sum_{i=0}^n p_i
=\sum_{i=0}^n q_i,
\qquad\text{hence}\qquad
\sum_{i=0}^n q_i(q_i-1)=0.
\]
Each summand is nonnegative because $q_i$ is an integer, and vanishes
only for $q_i\in\{0,1\}$. Thus every coefficient of $Q$ is zero or one,
so $Q\in\cP_n$.

Finally, canceling $A$ in $AB^*=AB$, and $B^*$ in $AB^*=A^*B^*$,
gives, respectively,
\[
Q=P\Longleftrightarrow B=B^*,\qquad
Q=P^*\Longleftrightarrow A=A^*,
\]
concluding the proof.
\end{proof}

We establish the following reciprocal-factor estimate by combining the separation argument of [BV19, proof of Lemma 40] with conditioning on the sums of opposite coefficients.

\begin{proposition}\label{lem:reciprocal}
There exist constants $c,C>0$ such that, for every
$n\ge1$,
\[
\Prob\left(
\begin{array}{c}
P_n\text{ has a monic reciprocal factor}\\
\text{that is not a product of cyclotomic polynomials}
\end{array}\right)
\le C\exp\!\left(-\frac{cn}{(\log n)^4}\right).
\]
\end{proposition}

To prove Proposition~\ref{lem:reciprocal}, we state a version of the separation argument in
\cite{BV19}*{proof of Lemma 40}. For a nonzero polynomial $F$,
write $M(F)$ for its Mahler measure.

\begin{lemma}\label{lem:BV-separation}
For all sufficiently large $n$ the following holds with $T=\log(n)^4$. For every monic noncyclotomic irreducible
polynomial $J\in\Z[x]$ with $J\ne x$ and $\deg J\le n$
there is a prime $q\in(T,2T]$ such that for every
$a\in\{0,\ldots,q-1\}$ no nonzero polynomial of the form
\[
D(x)=\sum_{\substack{0\le i\le n\\ i\equiv a\,\,\mathrm{mod} \,\, q}} d_i x^i
\qquad\text{with}\qquad d_i\in\{-1,0,1\}
\]
is divisible by $J$.
\end{lemma}

\begin{proof}
Let $\alpha_1,\ldots,\alpha_h$ be the roots of $J$.
For sufficiently large $h$, since $T\ge 4\log h$,
\cite{BV19}*{Lemma 26} gives a prime $q\in(T,2T]$
such that
$\alpha_1^q,\ldots,\alpha_h^q$ are distinct.
For the remaining bounded values of $h$, choose any prime
$q\in(T,2T]$ by Bertrand's postulate. For sufficiently large
$n$, we have $q>h^2+1$, which ensures distinctness as
the field $\Q(\alpha_i,\alpha_j)$ has degree at most $h^2$ over $\mathbb{Q}$
and cannot contain a primitive $q$th root of unity. So we conclude that in any case $\alpha_1^q,\ldots,\alpha_h^q$ are distinct.

Recall that Dobrowolski's bound \cite{Dob79} gives
$\log M(J)\geq (\log n)^{-3}$ for sufficiently large $n$, and hence
$\log(M(J)^q)=q\log M(J)\ge \log n$ and therefore
$M(J)^q>\sqrt{n+1}$ for sufficiently large $n$. Suppose that $J\mid D$ for a polynomial as in the statement.
Then we can express $D$ as $D(x)=x^aV(x^q)$ with
$V\in\Z[x]$. Since $J\ne x$, the polynomial $V$ vanishes at
the distinct numbers $\alpha_1^q,\ldots,\alpha_h^q$. Hence
\[
\sqrt{n+1}<M(J)^q
\le M(V)=M(D)\le\|D\|_2\le\sqrt{n+1},
\]
a contradiction.
\end{proof}

\begin{proof}[Proof of Proposition~\ref{lem:reciprocal}]
It suffices to consider sufficiently large $n$, increasing
$C$ to cover the remaining cases. Denote $T=(\log n)^4$ as in Lemma~\ref{lem:BV-separation} and write
\[
P_n(x)=\sum_{i=0}^n p_i x^i,\qquad
m=\left\lfloor\frac{n-1}{2}\right\rfloor,\qquad
S=P_n+P_n^*.
\]
The idea of the proof is that the reciprocal factors of $P_n$ are also factors of $S$. If $S$ is fixed, the coefficients $s_i=p_i+p_{n-i}$ for $1\le i\le m$ are determined.
Therefore, conditional on $S$ the pairs $(p_i, p_{n-i})$ with $s_i=1$ independently
take the values $(0,1)$ and $(1,0)$, each with probability
$1/2$ and all other coefficients with $s_i = 0$ or $s_i = 2$ are determined. We will condition on $S$, which fixes the irreducible
factors of $S$ while preserving independent choices of opposite
coefficient pairs. We then use this remaining randomness and
Lemma~\ref{lem:BV-separation} to bound the probability that
any noncyclotomic monic irreducible factor of $S$ divides $P_n$.

For each prime $q\in(T,2T]$, choose
$a_q\in\{0,\ldots,q-1\}$ with $2a_q\equiv n\,\, \mathrm{mod} \,\, q$ and set
\[
I_q=\{1\le i\le m:i\equiv a_q \,\, \mathrm{mod} \,\, q \text{  and } s_i=1\}.
\]
Note that $I_q$ is determined by $S$ and as $P_n$ is random the set $I_q$ is random too. Let $N_q$ be the number of indices $1\le i\le m$ satisfying
$i\equiv a_q\pmod q$. For sufficiently large $n$, we have
$N_q\ge n/(8T)$. As we haven't conditioned on $S$ yet, the corresponding
indicators $\mathbf 1_{\{s_i=1\}}$ are independent Bernoulli
variables of parameter $1/2$ as they consist of disjoint
coefficient pairs. Thus, as $I_q$ depends on $P_n$ and is therefore a random subset, it holds that $|I_q|\sim\operatorname{Bin}(N_q,1/2)$
and the lower-tail Chernoff bound gives 
\[
\Prob\!\left(|I_q|<\frac{n}{32T}\right)
\le \Prob\!\left(|I_q|<\frac{N_q}{4}\right)
\le \exp\!\left(-\frac{N_q}{16}\right)
\le \exp\!\left(-\frac{n}{128T}\right).
\]
Thus, a union bound over the at most $T$ primes
in $(T,2T]$ shows that the $S$-measurable event
\[
\mathcal G=
\{|I_q|\ge n/(32T)\text{ for every prime }q\in(T,2T]\}
\]
satisfies
\[
\Prob(\mathcal G^c)\le T\exp\!\left(-\frac{n}{128T}\right).
\]

Fix a value of $S$ for which $\mathcal G$ holds.
Since $\deg S=n$ and $S(0)=2$, there are at most $n$
monic irreducible divisors of $S$, none equal to $x$.
For each noncyclotomic irreducible factor $J$, choose $q$ by
Lemma~\ref{lem:BV-separation} and further condition on all pairs $(p_i, p_{n-i})$
with $i$ outside $I_q$. Note that the pairs $(p_i, p_{n-i})$ with $i \in I_q$ after fixing $S$ are all either $(1,0)$ or $(0,1)$. With the pairs outside $I_q$ fixed, at most one assignment of the
pairs inside $I_q$ can give $J\mid P_n$. Indeed, the difference of two
distinct assignments would be a nonzero polynomial divisible
by $J$, with coefficients in $\{-1,0,1\}$ and all nonzero
coefficients indexed by the residue class $a_q$ modulo $q$,
since $n-i\equiv i\equiv a_q\,\, \mathrm{mod}\,\, q$.
This contradicts Lemma~\ref{lem:BV-separation}.
Therefore,
\[
\Prob\bigl(J\text{ divides }P_n\mid S\bigr)
\le 2^{-|I_q|}\le 2^{-n/(32T)}.
\]

Every monic reciprocal divisor $H$ of $P_n$ also divides
$P_n^*$ and hence divides $S$. Write $H=J_1\cdots J_r$ with each
$J_j\in\mathbb Z[x]$ monic and irreducible. If $H$ is not a
product of cyclotomic polynomials, some $J_j$ is noncyclotomic
and divides both $S$ and $P_n$. As we condition on $S$, the polynomials that factor $S$ are determined. Therefore, conditional on $S$ satisfying
$\mathcal G$, a union bound over the at most $n$ such
irreducible divisors of $S$ gives probability at most
$n2^{-n/(32T)}$. Averaging over $S$ therefore bounds the
desired probability by
\[
\Prob(\mathcal G^c)+n2^{-n/(32T)}
\le C\exp\!\left(-\frac{cn}{(\log n)^4}\right).
\] This concludes the proof.
\end{proof}

Having established all these ingredients, we first prove Theorem~\ref{thm:main}.

\begin{proof}[Proof of Theorem~\ref{thm:main}]
By Gauss's lemma, if $P_n$ is reducible over $\mathbb{Q}$ it is
reducible over $\mathbb{Z}$ too. So Lemma~\ref{lem:flip} shows
that a reducible $P_n$ either has a non-trivial twin factorization
or a nonconstant monic reciprocal divisor. Hence Theorem~\ref{thm:rigidity} and
Proposition~\ref{lem:reciprocal} give, for every $A>0$,
\begin{equation}\label{eq:noncyclotomic}
\Prob(P_n\text{ reducible and without a cyclotomic divisor})
=O_A(n^{-A}).
\end{equation}

It remains to deal with the occurrence of a cyclotomic factor. Let $\Phi_k$ be the $k$th cyclotomic polynomial with $\deg \Phi_k = \varphi(k)$ for $\varphi$ Euler's totient function that counts integers $1\leq j \leq k$ with $\gcd(k,j) = 1$. The unconditional
cyclotomic estimate in \cite{BV19}*{Lemma 45 and Section 9.2} gives
\begin{equation}\label{eq:cyclotomic}
\Prob\bigl(\Phi_k\mid P_n\text{ for some }\varphi(k)\ge2\bigr)
=O(n^{-1}).
\end{equation}
The only remaining cyclotomic factors are $x-1$, which cannot divide
$P_n$ as $P_n(1) > 0$, and $x+1$. Thus, for $n\geq 2$, equations \eqref{eq:noncyclotomic}--\eqref{eq:cyclotomic} imply 
\[
0\le\Prob(P_n\text{ reducible})-\Prob(P_n(-1)=0)\ll n^{-1}.
\]

Note that for $m \geq 1$, we have $\Prob(P_{2m+1}(-1)=0)=2^{-2m}\binom{2m}{m}$ and $\Prob(P_{2m}(-1)=0)=2^{-(2m-1)}\binom{2m-1}{m+1}.$
Stirling's formula yields
$\Prob(P_n(-1)=0)=\sqrt{2/(\pi n)}+O(n^{-3/2})$,
completing the proof.
\end{proof}

\begin{remark}\label{rm:higherorderexpansion}
The argument of the proof of Theorem~\ref{thm:main} can be refined to obtain the asymptotic expansion of the probability of reducibility to any fixed order. The resulting coefficients may depend on the residue class of \(n\) modulo an integer depending on the desired order. Indeed, \eqref{eq:noncyclotomic} and the cyclotomic estimates in [BV19, Section 9.2] reduce the problem to finitely many cyclotomic divisibility events that one can combine with the inclusion–exclusion principle. For example, it can be shown that 
    \begin{align*}
        \Prob(P_n\text{ is reducible over }\Q)
=\frac{A}{n^{1/2}} + \frac{B}{n} + \frac{A(\delta_n - 2B)}{n^{3/2}} + O\left(\frac{1}{n^2}\right)
    \end{align*} for $A = \sqrt{2/\pi}$, $B = \frac{4(1 + \sqrt{3})}{\pi}$ and $\delta_n =\{ \begin{smallmatrix}
        -17/4, &\, n \text{ even} \\
1/4, &\, n \text{ odd}
    \end{smallmatrix}.$

\end{remark}

We conclude the paper with the proof of Theorem~\ref{thm:IrredFac}, for which we use the following lemma dealing with high degree divisors.

\begin{lemma}\label{lem:reciprocal-degree}
For $P\in\Z[x]$, let $\bar P\in\mathbb F_2[x]$ denote its
coefficientwise reduction modulo two. For every $n\ge1$ and
integer $L\ge1$,
\begin{equation}\label{eq:reciprocal-tail}
\Prob\bigl(\deg\gcd(\bar P_n,\bar P_n^*)\ge L\bigr)
\le 8\,2^{-L/2},
\end{equation}
where the greatest common divisor is taken in $\mathbb F_2[x]$.
In particular, the probability that $P_n$ has a monic reciprocal
divisor in $\mathbb{Z}[x]$ of degree at least $L$ is at most $8\,2^{-L/2}$.
\end{lemma}

\begin{proof}
As reversal exchanges the two polynomials whose gcd we are
taking, their monic greatest common divisor is reciprocal.
There are $2^{\lfloor h/2\rfloor}$ monic reciprocal polynomials
of degree $h$ and constant coefficient one over $\mathbb F_2$.

Fix such a polynomial $R$ of degree $h\le n$. A polynomial in
$\mathcal B_n$ divisible by $R$ is uniquely determined by its
coefficients of $x^h,\ldots,x^{n-1}$. Indeed, the difference of
two such polynomials agreeing in these coefficients has degree
less than $h$ and is divisible by $R$, so must vanish. Hence 
$\Prob(R\mid\bar P_n) \le \frac{2^{n-h}}{2^{n-1}}=2^{1-h}.$ Summing over $h\ge L$ gives
\[
\Prob\bigl(\deg\gcd(\bar P_n,\bar P_n^*)\ge L\bigr)
\le \sum_{h=L}^{\infty}2^{\lfloor h/2\rfloor}2^{1-h}
\le 6\cdot 2^{-L/2}.
\]
Finally, if $H$ is a monic reciprocal divisor of $P_n$, then
$\bar H$ has the same degree and divides both $\bar P_n$ and
$\bar P_n^*$.
\end{proof}

\begin{proof}[Proof of Theorem~\ref{thm:IrredFac}]
Let $\Phi$ be the product of all cyclotomic factors of $P_n$,
counted with multiplicity, and put $I=P_n/\Phi$. By Theorem~\ref{thm:rigidity} and
Proposition~\ref{lem:reciprocal}, outside an event of probability
at most $C\exp(-\frac{cn}{(\log n)^4}),$ the polynomial $P_n$ has neither a non-trivial twin factorization
nor a monic reciprocal divisor that is not a product of
cyclotomic polynomials. Denote this event of high probability
by $\mathcal E$.

On $\mathcal E$, we claim that $I$ is either $1$ or irreducible.
Otherwise, by Gauss's lemma, write $I=AB$ with $A,B\in\Z[x]$
monic and nonconstant. Since $I$ has no cyclotomic factors,
neither $A$ nor $\Phi B$ is a product of cyclotomic polynomials.
Consequently, neither is reciprocal on $\mathcal E$.
Applying Lemma~\ref{lem:flip} to $P_n=A(\Phi B)$ gives a
non-trivial twin factorization, a contradiction.

Since $\Phi$ is reciprocal, Lemma~\ref{lem:reciprocal-degree} therefore gives
for every integer $1\le L\le n$ that $\Prob(\deg \Phi\ge L)\le 8\,2^{-L/2}.$ On $\mathcal E\cap\{\deg \Phi<L\}$, the polynomial $I$ is
nonconstant and hence irreducible and noncyclotomic. Thus the
required factorization, with $\deg \Phi<L$, holds with probability
at least
\[
1-C\exp\!\left(-\frac{cn}{(\log n)^4}\right)-8\,2^{-L/2}.
\]
Taking $L=n$ proves the first assertion. Taking
$L=\lfloor n^{\alpha}\rfloor+1$ proves the second with suitable constants $c_{\alpha}, C_{\alpha} > 0$ depending on $\alpha$.
\end{proof}